\documentclass{article}
\usepackage{graphicx} 
\usepackage{mathtools, amssymb, amsthm, mathabx, titlepic, dynkin-diagrams}
\usepackage[
    backend=biber,
    style=numeric,
]{biblatex}
\graphicspath{ {./Images/} } 
\usepackage[indent]{parskip}

\theoremstyle{definition}
\newtheorem{definition}{Definition}[section]
\newtheorem{theorem}{Theorem}
\newtheorem{example}[theorem]{Example}
\newtheorem{lemma}[theorem]{Lemma}

\newcommand{\func}[1]{\operatorname{\mathrm{#1}}}
\newcommand{\fix}[2]{{#1}^{#2}}
\DeclareMathOperator{\stab}{Stab}
\DeclareMathOperator{\conj}{Conj}
\DeclareMathOperator{\id}{e}
\DeclareMathOperator{\dih}{Dih}

\DeclareMathOperator{\Z}{Z}

\newcommand{\Conj}[2]{\conj(#1; #2)}
\newcommand{\C}[2]{\func{C}(#1; #2)}

\newcommand{\cconj}[2]{\card{\conj(#1; #2)}}
\newcommand{\cC}[2]{\card{\C{#1}{#2}}}
\newcommand{\cstab}[1]{\card{\stab(#1)}}
\newcommand{\refl}[0]{s_1, s_2, ..., s_n}
\newcommand{\card}[1]{\left| {#1} \right|}

\newcommand{\numfix}[1]{\card{\fix{X}{#1}}}

\def\acts{\mathrel{\reflectbox{$\righttoleftarrow$}}}

\title{Algebraic Enumeration of Polypolyhedra}
\author{G.\! Henderson-Walshe, M.\! Langton, J.C.\! McLeod, P.L.\! Wilson}

\begin{document}

\maketitle

\begin{abstract}
Polypolyhedra are edge-transitive compounds of polyhedra. In this paper we use group theory to determine the number of distinct polypolyhedra whose symmetry group is any given finite irreducible Coxeter group. We apply this result in order to enumerate the 3-dimensional polypolyhedra.
\end{abstract}

\section{Introduction}\label{sec:intro}

\emph{Polypolyhedra} are edge-transitive compounds of polyhedra, originating as origami models. They were first systematically described and enumerated by R.\! Lang in \cite{lang_2016}.
His article framed the problem of enumerating the polypolyhedra and divided it into steps. He gave a definition of polypolyhedra and recognised the two criteria by which any two polypolyhedra can be determined to be distinct or equivalent: first by symmetry, then by topology. Lang's methodology for enumerating polypolyhedra was computational and essentially brute force; we favour instead an analytical approach and point out a mistake in Lang's algorithm herein. Only one paper concerning polypolyhedra has been published since Lang's, namely \cite{belcastro}. That paper enumerated symmetric colourings of polypolyhedra, which is not the purpose of the present paper. 

The present paper enumerates the polypolyhedra by symmetry without resorting to a computer search. We frame the problem in more mathematical terms rather than with reference to origami, but we inherit Lang's division of the problem into steps, and confirm his results. We derive a new, algebraic method of enumerating the polypolyhedra based on symmetry, which extends to higher dimensional analogues of polypolyhedra.

Although our enumeration produces the same results as Lang's, we find that Lang's method is based on an assumption which is not true, but which happened to lead him to the correct answer in the 3-dimensional case. We explain his algorithm and its mistake in Section \ref{sec:langsmain result}.

The present section lays out the mathematical background for our main result; Section \ref{sec:back} develops the tools needed for our main result; Section \ref{sec:main result} states our main result and the data produced therein. Our conclusions form Section \ref{sec:conc}.

\subsection{Key concepts}\label{key_concepts}

Lang's definition of polypolyhedra is stated somewhat informally. We give an equivalent definition of what we call \emph{strict polypolyhedra}, stated without reference to origami. We define a \emph{general polypolyhedron} as a set of straight line segments which we call \emph{struts}, which connect a pair of points called \emph{endpoints}, satisfying the following two properties.
\begin{enumerate}
    \item There is a finite group of rotations about the origin which is transitive on the set of struts.
    \item Each endpoint must be incident at least two struts.
\end{enumerate}
The following additional properties define a strict polypolyhedron.
\begin{enumerate}
    \setcounter{enumi}{2}
    \item Each strut has precisely two points of intersection with other struts. These points of intersection are called \emph{endpoints}.
    \item The struts are not all connected by a sequence of endpoints.
    \item The curves defined by the union of the struts are topologically linked.
\end{enumerate}

\subsection{Overview of the enumeration strategy}

Property 1 gives polypolyhedra a very simple algebraic structure, allowing us to capture all of the important information about a polypolyhedron in few variables. The enumeration of polypolyhedra follows by reducing polypolyhedra to a classification based on Property 1 and enumeration of each classification. Property 2 is merely a restriction by which we filter out some possible polypolyhedra.

Enumeration depends on drawing distinction between objects. Determining whether two polypolyhedra are distinct requires an equivalence relation, which is defined in terms of orbits under a group action.

A group action $G \acts X$ is a representation of a group $G$ as permutations on a set $X$, or in our case transformations of a vector space. The origin-preserving isometries of a vector space are an example of a group action.

The choice of group is the most fundamental property of a polypolyhedron. Previous origami-based analysis only used the finite rotation groups of $\mathbb{R}^3$, whereas our main result is more abstract and can be applied with any finite isometry group. 

Since the group must be transitive on the set of struts, we need only specify the location of one \emph{seed strut}, with the remaining struts being \emph{generated} by taking the images of the seed strut under each action of the group. This bijection between a polypolyhedron and its seed strut means that two seed struts are equivalent (in the same orbit) precisely when the polypolyhedra they generate are equivalent. Therefore, our first task is to calculate the number of orbits of (seed) struts.

A seed strut can be categorised (though not specified) by the orbits of its endpoints under the group action, a pair which we call the \emph{orbit-type}. Our main result, Theorem \ref{thm:main result}, computes the number of distinct polypolyhedra of a given orbit-type and group action as the number of orbits of possible seed struts under the group. 

The main new result is the calculation of $\fix{X}{w}$, the number of struts fixed by each group element $w$. Below we list the steps needed to calculate the formula for the number of general polypolyhedra of a given orbit-type, 
\begin{enumerate}
    \item Compute the sizes of certain conjugacy classes of the symmetry groups of the polypolyhedron $W$, $W_I$, and $W_J$.
    \item For each conjugacy class choose a representative, $w\in W$.
    \item Compute the number of possible endpoints of each orbit $C_I$ and $C_J$ respectively fixed by $w$ using Theorem \ref{thm:important}.
    \item Compute the number of possible $(C_I,C_J)$ struts fixed by $w$.
    \item Compute the number of distinct struts using Burnside's lemma.
    \item Subtract the number of univalent struts.
\end{enumerate}
 
 This allows us to calculate the number of general polypolyhedra. Algebraic methods may be inadequate for the enumeration of strict polypolyhedra, which we discuss in the conclsuion.

\subsection{Definitions and notation}
Two notions are central to the present paper: that of a group action and that of a conjugacy action. A group action is a group $W$ and a set $X$---denoted $W \acts X$---such that $W$ permutes the elements of $X$. Relevant to our purposes, a group of matrices is a group action on a vector space and the symmetric group $S_n$ is a group action on a set of $n$ elements. The subset of $W$ which fixes some $x \in X$ is a subgroup called the \emph{stabiliser}, $\stab(x)$.

A conjugacy action is a group action of a group on itself, $W \acts W$, defined by $w: u \mapsto wuw^{-1}$. The centraliser $\func{C}(w)$ is the set $\{u \in W \mid wu=uw\}$, which is the same as the stabiliser for a general group action. A conjugacy class of $w$, $\Conj{W}{w}$, is the set of images of $w$ under the conjugacy action. To know the size of a conjugacy class is to know the size of its centraliser, by the following well-known lemma \cite{gallian}.

\begin{lemma}\label{lem:centraliser conjugacy}
\[
    \card{W} = \cconj{W}{w} \cC{W}{w}
\]
\end{lemma}

Table \ref{tab:notation} lists notation used in this paper. For the purpose of this table, $I, J \subset S$, $H < W$, $w \in W$, and $x \in X$.
\begin{table}[h]
    \begin{center}
    \begin{tabular}{|c|c|c|}
        \hline
         $W \acts X$ & The symmetry group of the polypolyhedron, acting on the set $X$ \\
         \hline
         $W \acts W$ & The conjugacy action $w: u \mapsto wuw^{-1}$ \\
         \hline
         $S$ & The simple reflections generating $W$, $\{\refl \}$ \\
         \hline
         $n$ & The rank (number of simple reflections) of $W$, $\card{S}$ \\
         \hline
         $W^*$ & The set of rotations in $W$ \\
         \hline
         $\Conj{H}{w}$ & The conjugacy class of $w$ in $H$ \\
         \hline
         $\C{H}{w}$ & The centraliser of $w$ in $H$, $\{u \in H \mid wu=uw\}$ \\
         \hline
         $\Z(W)$ & The centre of $W$, $\{w \in W \mid \forall u \in W,~ wu=uw\}$ \\
         \hline
         $G(W)$ & The Coxeter graph of $W$ \\
         \hline
         $W_I$ & For $I \subset S$, $W_I$ is the group generated by $I$ \\
         \hline
         $\stab(x)$ & $\{w \in X \mid w(x)=x \}$ \\
         \hline
         $X^{w}$ & $\{x \in X \mid w(x)=x \}$ \\
         \hline
         $X/W$ & The orbits of $X$ under $W$ \\
         \hline
         $W/H$ & The set of cosets of $H$ in $W$ \\
         \hline
         $[W:H]$ & The index of $H$ in $W$, $\card{W/H}$ \\
         \hline
         $f_I(w)$ & The size of $\{ u \in W/W_I \mid w(uC_I) = uC_I \}$ \\
         \hline
         $l(w)$ & The length of $w$, defined in Section \ref{sec:cox}  \\
         \hline
         $P(W,I,J)$ & The number of distinct general polypolyhedra of a given orbit-type \\
         \hline
    \end{tabular}
    \caption{Notation used in the paper.}\label{tab:notation}
    \end{center}
\end{table}

\section{Background}\label{sec:back} 
\subsection{Coxeter Groups}\label{sec:cox}

Our equivalence relation will be based on a group of isometries which preserve the origin. This group of isometries forms a subgroup of the orthogonal group of order $n$, $O(n) = \{Q \in GL(\mathbb{R} \mid QQ^T = 1\}$. Every group we use in this paper is origin-preserving, so for the sake of brevity we will not state that explicitly from now on. One kind of group element is a \emph{reflection}. Such an element is a matrix, $R$, with one eigenvalue equal to $1$ and with multiplicity $n-1$, and another eigenvalue equal to $-1$ with multiplicity $1$. The orthogonal complement of the eigenvector with eigenvalue $-1$ is clearly fixed, so it must be the eigenspace of $1$. Therefore the two eigenspaces, $E_{1}(R)$ and $E_{-1}(R)$, are orthogonal complements, meaning that $R \in O(n)$.

Coxeter groups correspond to subgroups of $O(n)$ generated by a subset of the reflections. This is essentially a consequence of the following lemma.

\begin{lemma}
    Every $M \in O(n)$ can be expressed as a product of at most $n$ reflections.
\end{lemma}
\begin{proof}
    Denote by $e_i$ the $i$th standard basis vector for $\mathbb{R}^n$. For each $1<k\leq n$, define $M_k$ inductively as $M_k=R_{k}R_{k-1}...R_1$, where $R_{k}$ is the reflection which sends $M_k(e_{k})$ to $M(e_{k})$. Further, define $M_0 = I_n$. We will show that $M_n = M$. 
    
    Assume $M_{k-1}(e_i) = M(e_i)$ for all $i < k$. If $M_{k-1}(e_k) = M(e_k)$ then $R_k = I_n$ and $M_k(e_i)= M(e)$ for all $i <= k$. If not, then since $R_k(M_{k-1}(e_{k})=M(e_{k})$ and $R_k^2 = I_n$, $R_k(M_{k-1}(e_k)-M(e_k)) = -(M_{k-1}(e_k) - M(e_k)) \neq 0$. Hence $E_{-1}(R_k) = \func{span}\{M_{k-1}(e_k) - M_{k-1}(e_k)\}$, which is orthogonal to $M_{k-1}(e_i)$ for all $i \neq k$. Therefore, $M_{k-1}(e_i) \in E_1(R_k)$ for all $i < k$. Hence $M_k(e_i) = M(e_i)$ for all $i <= k$.
    
    By induction $M_n(e_i) = M(e_i)$ for all $e_i$, so $M_n$ is a product of reflections equal to $M$. As a result, the finite subgroups of $O(n)$ by which we compare polypolyhedra are generated by finite sets of reflections.
\end{proof}

As stated above, this result provides a simplifying way to view these groups, in the language of Coxeter theory \cite[p.~31]{wilson}, as follows.

Any group defined in terms of a set of generating reflections is called a Coxeter group. A description of a group in terms of a list of generators and relations between them is called a \emph{presentation}. The following is a presentation of a general Coxeter group, with the list of generators on the left and the relations on the right. The relations are described by a function $m(i, j)$, which is the order of the element $s_i s_j$.
\begin{equation*}
    W = \langle s_1, s_2, ..., s_n \mid (s_i s_j)^{m(i,j)} = 1 \rangle
\end{equation*}
In this presentation, $m(i,j)=1$ if and only if $i = j$, otherwise $m(i,j) \geq 2$. It is trivial to show that $m(i,j)=m(j,i)$. The set $S= \{s_1, s_2, ..., s_n\}$ should be a minimal set needed to generate $W$ and is referred to as the set of \emph{simple reflections}. The number of simple reflections is denoted $n$, which is also least rank of a matrix representation of the group, so we refer to $n$ as the rank.

A Coxeter group $W$ can be represented by a finite graph whose edges are labelled with integers $\geq 3$ called a \emph{Coxeter graph}, or $G(W)$. This graph is constructed by mapping each simple reflection to a unique node and adding an edge between the nodes corresponding to $s_i$ and $s_j$ with label $m(i,j)$ if $m(i,j) \geq 3$. Since all edges are labelled and the label $3$ is most frequent, the label of $3$ may be omitted \cite[p.~31]{humphrey}. 

\begin{example}\label{eg:h3}
Icosahedral symmetry, denoted $H_3$, is the isometry group of a regular icosahedron. It has the following presentation as a Coxeter group:
\begin{equation*}
    H_3 = \langle s_1, s_2, s_3 \mid s_1^2 = s_2^2 = s_3^2 = (s_1 s_2)^5 = (s_2 s_3)^3 = (s_3 s_1)^2 = 1 \rangle.
\end{equation*}
It has Coxeter graph \dynkin[Coxeter]H3. Each pair of reflections produce a distinct rotation of an icosahedron. $s_1 s_2$ is a vertex rotation, hence it has order 5; $s_2 s_3$ is face rotation, hence it has order 3; and $s_1 s_3$ is an edge rotation, hence it has order 2.
\end{example}

A Coxeter group $W$ is called irreducible precisely when $G(W)$ is connected \cite[p.~30]{humphrey}. The irreducible Coxeter groups have been classified and shown to correspond precisely to the finite isometry groups which fix the origin in $\mathbb{R}^n$ \cite{coxeter_enum}. Although our main result applies to any group, it requires that one know the size of the centraliser of an arbitrary element. 

The present paper only provides a method for finding centralisers for the series of groups $\alpha_n$, $\beta_n$, $\dih_d$ and $H_3$, which constitute ``nearly all''\footnote{Similarly to the classification of simple groups, irreducible Coxeter groups come in three series, which we cover, and a few exceptional groups, of which we only cover $H_3$.} irreducible Coxeter groups. We define these groups in Table \ref{tab:irr class} by their Coxeter graph. We have also included an abstract group to which it is isomorphic, which we will need in Section \ref{sec:conjgacy}. 
\begin{table}[ht]
\begin{center}
\begin{tabular}{|c|c|c|}
\hline
Notation & Graph & Abstract group \cite[p.~33]{wilson}\\
\hline
$\alpha_n$ & \dynkin[Coxeter] A{} & $S_{n+1}$, symmetric group\\
\hline
$\beta_n$ & \dynkin[Coxeter] B{} & $S_2 \wr S_{n+1}$, wreath product of symmetric groups \\
\hline
$\dih_d$ & \dynkin[Coxeter,gonality=d] I{} & $ S_2 \ltimes C_n$, semidirect product with cyclic group\\
\hline
$H_3$ & \dynkin[Coxeter] H3 & $ S_2 \times A_5$, direct product with alternating group\\
\hline
\end{tabular}
\caption{Description of the Coxeter groups we use in this paper.}\label{tab:irr class}
\end{center}
\end{table}

\subsection{The length function}

By expressing group elements as products of a fixed set of generators, we get the following map which will be needed to define the stabilisers of endpoints of polypolyhedra. Let $l: W \rightarrow \mathbb{Z}$ be the minimum number of simple reflections for which a group element $w$ can be expressed as a product. 

The kernel either of the determinant homomorphism or the homomorphism $w \mapsto (-1)^{l(w)}$ both give the index-$2$ subgroup of $W$ which contains all of the rotations in $W$, which we denote $W^*$. The subgroup $W^*$ is what we will use to ``generate'' the polypolyhedra, explained in Subsection \ref{subsec:orbit_type}.

The Coxeter group presentation also allows us to express a type of subgroup which will be useful in classifying the polypolyhedra. A \emph{standard parabolic subgroup}, $W_I$, is the group generated by only the reflections of $I \subseteq S$. Its conjugates are called \emph{parabolic subgroups}. For example, if $W =\beta_4$ and $I= \{s_1, s_3, s_4\}$, $W_I$ is obtained by removing $s_2$, like so: \dynkin[Coxeter] B{*x**}. The result is the group \dynkin[Coxeter] A1 \dynkin[Coxeter] B2, or $S_2 \times \dih_4$.

\subsection{Facets}\label{subsec:orbits}

The \emph{orbit} of an element of $X$ is the set of its images under $W \acts X$. The group structure ensures that this is an equivalence relation, which we show in the following lemma. 

\begin{lemma}
    The relation $y \sim x$ defined by $y = w(x)$ for some $w \in W$ is an equivalence relation on $X$.
\end{lemma}
\begin{proof}
    Let $x, y, z \in X$ and $u, v \in W$. If $x \sim y$, $y = w(x)$, so $x = a^{-1}(y)$. Thus $y \sim x$ and so $\sim$ is symmetric. If $z \sim y$, $z = v(y)$ then $z = v(u(x))=vu(x)$, so $z \sim x$, meaning $\sim$ is transitive. Since $x = \id(x)$, $x \sim x$, and so $\sim$ is reflexive. Therefore $\sim$ is an equivalence relation.
\end{proof}

Thus we refer to \emph{the} orbits as a partition of $X$.

We now give a useful classification of the points in $\mathbb{R}^n$ under the action of $W$. Let $S$ be a set of simple reflections and $I \subset S$. Noting that all reflections $s$ fix a hyperplane through the origin, let $H_i^0=H_i$ be the hyperplane corresponding to $s_i$. The reflection $s_i$ maps the negative half space $H_{i}^-$ to the positive half-space $H_{i}^+$ and vice versa. This gives us the following characterisation \cite{humphrey}:
\[
    C_I = \bigcap_{i} H_i^{\epsilon_i} \text{, where $\epsilon_i = 0$ if $s_i \in I$ and $+$ if $s_i \notin I$}
\]
Because $W$ is generated by $S$ and this characterisation describes the action of $S$ on $C_I$, the action of $W$ on $C_I$ is fully determined. As we shall see, the stabiliser of $C_I$ is $W_I$. Therefore the orbit of $C_I$ is composed of each $wC_I$, which are called the \emph{facets} of type $I$. Since $wC_I = C_I \implies w \in \stab(C_I)=W_I$, the orbit corresponds to the cosets $W / W_I$, which is the orbit of the facet. Every point lies in the orbit of some $C_I$, so we have an abstract interpretation of $C_I$.

The intersection of $k$ linearly independent hyperplanes $H_i$ through the origin is a subspace of dimension $n - k$. Since our set of hyperplanes come from $I \subseteq S$, and $S$ corresponds to a linearly independent set of hyperplanes $\dim C_I = n - |I|$. 

If $|I| = n$ then $C_I$ is just the origin, so to specify a non-trivial single point, we need to take $|I| = n-1$, which gives us an axis. All points on an axis are identical in terms of how they are acted upon by the group. Therefore, we can view the axis as a positive and negative pair of a pre-determined radius. 

The facets can also be named according to the corresponding facets of polyhedra under polyhedral symmetry, though these names are comparatively less useful. As polyhedral groups have three generating reflections, there are three subsets $I \subset S$ with $\card{I} = 2$. These happen to be the vertices (V), the edges (E), and the faces (F). The empty subset are called $C$ points by Lang.

The group by which we compare struts for equivalence is not exactly the same as that by which we generate the polypolyhedron from the seed strut. This is because including the reflections of $W$ in the generation of a polypolyhedron would create an extra intersection for any reflection plane passing through and not containing the seed strut. On the other hand, we would not like polypolyhedra which are simply reflections of each other to be considered distinct. 

\subsection{Orbit-type and Generation of polypolyhedra}\label{subsec:orbit_type}
The seed strut is defined by an unordered pair of endpoints, which we shall classify by the facet types $C_I$ and $C_J$ of the endpoints. This is called the \emph{orbit-type}, which we shall write as $(W,I,J)$. If $I=J$ the type is called ``homoorbital'' while if $I \neq J$ then the type is called ``heterorbital''. 

A polypolyhedron is generated from its seed strut, $\{a,b\}$ as the set, $\{\{r(a), r(b)\} \mid r \in W^*\}$. In English, a polypolyhedron is the set of images of the seed strut under $W^*$. 

\subsection{The Conjugacy Action}

The following theorem gives a correspondence between the orbits of $W \acts X$ and the conjugacy action $W \acts W$ \cite{wilson}.

\begin{lemma}[Group action correspondence]\label{conjstab}
    The action of the group on a point $x$ corresponds to the conjugacy action on $\stab(x)$.
    If $H < W$ and $u \in W$, $\fix{X}{u H u^{-1}} = u(\fix{X}{H})$ and $\stab(u(X)) = u \stab(X) u^{-1}$.   
\end{lemma}
\begin{proof}
    Let $a, b \in X$, $b=u^{-1}(a)$. For any $\tau \in W$, let $w = u \tau u^{-1}$. Then $w(a) = u \tau(b)$, so $w(a)=a, a \in \fix{X}{w}$ if and only if $\tau(b) = b, b \in \fix{X}{\tau}$. Therefore $a \in \fix{X}{u H u^{-1}}$ if and only if $b \in \fix{X}{H}$, i.e. $a \in u(\fix{X}{H})$, so $\fix{X}{u H u^{-1}} = u(\fix{X}{H})$. Similarly, $a \in \fix{X}{\tau} \implies \tau \in \stab(u(X))$ and $b \in \fix{X}{w} \implies \tau \in u \stab(X) u^{-1}$, so $\stab(u(X)) = u \stab(X) u^{-1}$.
\end{proof}

\begin{theorem}\label{thm:parabolic_stab}
    If $\card{I}=n-1$ then $\stab(C_I) = \langle I \rangle$.
\end{theorem}
\begin{proof}

    Because for each $s \in I$, $s(C_I) = C_I$, it is clear that  $\langle I \rangle \leq \stab(C_I)$.  We will need the fact that $S = I \cup \{t \}$.

    We define $l_T(w): W \mapsto \mathbb{Z}$ as the minimum number of elements in the set $T$. The automorphism $u \mapsto wuw^{-1}$ maps the generating set $S$ to $wSw^{-1}$ and therefore $l_S(u) = l_{wSw^{-1}}(wuw^{-1})$. To show that for all $w \in \stab(hC_I)$, $w \in \langle hIh^{-1} \rangle$, we use induction on $l_{hSh^{-1}}(w)$.

    The base case is $l_S(w)=1 \implies w \in S$. By definition, the only simple reflections which fix $C_I$ are those in $I$. 

    If for any $s \in I$, $l_S(ws) < l_S(w)$ then by the Exchange Condition \cite{humphrey}, $w$ can be written as a product $u s$ for some $u \in W$. Therefore, since $s \in \stab(C_I)$, $u \in \stab(C_I)$. Further, $l_S(u) < l_S(w)$, so by the induction hypothesis, $u, s \in \langle I \rangle$. Hence $u s \in \langle I \rangle$. The same argument follows for $l(sw) < l(w)$.
    
    If we assume $l(ws)=l(sw) > l(w)$, then $w$ can only be written in a reduced expression as $w = tut$. Therefore $w(C_I) = tut(C_I)$ and so $u(tC_I) = t(w(C_I)) \implies u(tC_I) = tC_I$, i.e. $u \in \stab(tC_I)$. Because $l_{tSt}(w) = l(u) < l(w)$, we know that $w \in \langle tIt \rangle$. Therefore $u \in \stab(C_I)$, and so $w \notin \stab(C_I)$.
\end{proof}

\subsection{Fixed Points of $w$}
This section will explain our formula for calculating the size of the fixed set of group action $w \in W \acts X$ where $X$ is restricted to a single orbit. In other words, we calculate $\fix{X}{w}$ for a transitive group action.
\begin{lemma}\label{conjugacy} Conjugation preserves the number of fixed elements.
\begin{itemize}
    \item Conjugate elements fix the same number of elements of an orbit. That is, for all $v \in \Conj{W}{u}$, we have that $\card{\fix{X}{u}} = \card{\fix{X}{v}}$.
    
    \item Any two points in the same orbit are fixed by the same number of elements of each conjugacy class. That is, for all $v \in W$, we have: \[\cconj{\stab(x)}{u} = \cconj{\stab(v(x))}{u}\]

\end{itemize}
\end{lemma}
\begin{proof}
Let $u, v, w \in W$ such that $ wu = vw$, and  suppose $W \acts X$ is transitive.
\begin{align*}
    x \in \fix{X}{u} & \implies wu(x)=w(x) &\implies vw(x) = w(x) &\\
    & \implies w(x) \in \fix{X}{v} &\implies \card{\fix{X}{v}} \geq \card{\fix{X}{u}}. &\\
    \text{Conversely,} \\
    x \in \fix{X}{v} & \implies w^{-1}v(x)=w^{-1}(x) &\implies uw^{-1}(x) = w^{-1}(x) &\\
    & \implies w^{-1}(x) \in \fix{X}{u} &\implies \card{\fix{X}{u}} \geq \card{\fix{X}{v}}.
\end{align*}

Therefore, $\card{\fix{X}{u}} = \card{\fix{X}{v}}$.

By Lemma \ref{conjstab},
\begin{equation*}
    u \in \cconj{\stab(x)}{u} \iff v \in \cconj{w\stab(x)w^{-1}}{u}
\end{equation*}
Therefore, $\cconj{\stab(x)}{u} = \cconj{\stab(v(x))}{u}$.
\end{proof}

Lemma \ref{conjugacy} allows us to prove the following theorem for the calculation of the fixed sets of group actions.
\begin{theorem}[Size of Fixed Set Formula]\label{thm:important}
If $\cC{\stab(x)}{w} = 0$ then $\numfix{w} = 0$. Otherwise,

\[
    \numfix{w} = \frac{\cC{W}{w}}{\cC{\stab(x)}{w}}
\]

\begin{proof}
    By counting the number of pairs $(x,u)$ where $x \in X$ and $u \in \conj(w)$ such that $u(x)=x$ in two different ways, we get the  following equality:
    \[
        \sum_{u \in \conj(w)} \numfix{w} = \sum_{x' \in X} \cconj{\stab(x')}{w}.
    \]
    We apply Lemma \ref{conjugacy} on the left and right to obtain,
    \begin{equation*}
        \cconj{W}{w} \numfix{w} = [W:\stab(x)] \cconj{\stab(x)}{w} 
    \end{equation*}
    If $\cC{\stab(x)}{w} = 0$ then $\numfix{w} = 0$ since $\cC{W}{w} \neq 0$. Otherwise,
    \begin{align*}
        \frac{\card{W}}{\cC{W}{w}} \numfix{w} & = \frac{\card{W}}{\cstab{x}} \cconj{\stab(x)}{w} \\
        \implies \numfix{w} & = \frac{\cC{W}{w}}{\cstab{x}} \frac{\cstab{x}}{\cC{\stab(x)}{w}} \\
        \implies \numfix{w} & = \frac{\cC{W}{w}}{\cC{\stab(x)}{w}} \
    \end{align*}
\end{proof}
\end{theorem}
As an immediate consequence of this theorem, we have:
\[
    f_I(w) = \frac{\cC{W}{w}}{\cC{W_I}{w}}.
\]
This allows us to calculate the number of polypolyhedral endpoints fixed by each group element.

\subsection{The fixed set of struts}\label{sec:fixed struts}

With Theorem \ref{thm:important} we can compute $f_I(u)$, which will allow us to calculate $\card{\fix{X}{u}}$, the number of $(W,I,J)$ struts fixed by $u$. One may wonder why we do not apply Theorem \ref{thm:important} to the strut directly. The reason is that the struts do not necessarily form a single orbit, so their stabilisers might not form a single conjugacy class; in fact, if they did form a single conjugacy class then there could only ever be one distinct polypolyhedron per orbit-type, which would make this paper rather brief. A strut is fixed by an action $u$ either if both its endpoints are fixed or if the endpoints are swapped, or \emph{flipped}.

In the heteroorbital case, the endpoints may not be swapped by $u$, since they are not in the same orbit. So, we simply take all the possible ways of connecting a strut from the set of fixed $C_I$ points to the set of fixed $C_J$ points, the product $f_I(I) \cdot f_J(u)$.

In the homoorbital case, the number of ways to form a strut between fixed endpoints is $\binom{f_I(u)}{2}$. We then take the number of points which are swapped and divide by two to get the number of flipped struts, $\frac{1}{2} \left(f_I(u^2) - f_I(u) \right)$. Therefore, the total number of fixed struts is given by:
\begin{equation} \label{eq:endpointsToStruts}
    \card{\fix{X}{u}} = \binom{f_I(u)}{2} + \frac{f_I(u^2) - f_I(u)}{2} = \frac{f_I(u)\big(f_I(u) - 2\big) + f_I(u^2)}{2}
\end{equation}

Combining these results with Theorem \ref{thm:important}, we have the following equations for $\card{\fix{X}{u}}$ when $\cC{W_I}{u} \neq 0$.
\begin{theorem} \label{theorem:fixedstruts}
    
    For heterorbital types, 
    \[
    \card{\fix{X}{u}} = \frac{\cC{W}{u}^2}{\cC{W_I}{u} \cC{W_J}{u}}.
    \]
    For homoorbital types, we have
    \[
    \card{\fix{X}{u}} = 
    \frac{ \cC{W}{u}}{ 2\cC{W_I}{u} }\left(\frac{\cC{W}{u}}{\cC{W_I}{u}} - 2 \right) + \frac{\cC{W}{u^2}}{2\cC{W_I}{u^2}}.
    \]
    When $\cC{W_I}{u} = 0$ and $\cC{W_I}{u^2} \neq 0$, $f_I(u)=0$, so 
    \[
        \card{\fix{X}{u}} = \frac{\cC{W}{u^2}}{2\cC{W_I}{u^2}}.
    \]
\end{theorem}

The following lemma allows us to determine precisely which orbit-types fix no struts at all, and can therefore be ignored to simplify calculations.
\begin{lemma}[Filtering of Conjugacy Classes]\label{prop:shortcut}
    For heteroorbital types, 
    \[
        \Conj{W_I \cap W_J}{u} \neq \emptyset \iff \fix{X}{u} \neq \emptyset.
    \]
    For homoorbital types, 
    \[
    \Conj{\sqrt{W_I}}{u} \neq \emptyset \iff \fix{X}{u} \neq \emptyset \text{, where } \sqrt{W_I} = \{u \in W \mid u^2 \in W_I\}.
    \]
\end{lemma}
\begin{proof}

    For the heterorbital case, first note that $\fix{X}{u} \neq \emptyset \iff f_I(u) > 0$ and $f_J(u)>0$. $\Conj{W_I \cap W_J}{u} = \emptyset \implies \Conj{W_I}{u} = \emptyset$ without loss of generality. $W_I = \stab(C_I)$, so by Theorem \ref{conjstab}, $\Conj{\stab(vC_I)}{u} = \Conj{\stab(C_I)}{u} \neq \emptyset$ for all $v \in W$. This is equivalent to the statement $f_I(u) = 0$.

    For the homoorbital case, 
    \[ \Conj{\sqrt{\stab(C_I)}}{u} = \emptyset \iff \Conj{\sqrt{\stab(vC_I)}}{u} = \emptyset \text{ for all } v \in W,\] again by Theorem \ref{conjstab}. Therefore $u^2 \notin \stab(vC_I) = vW_Iv^{-1}$ and consequently $u \notin vW_Iv^{-1}$ as well. Equivalently, $f_I(u) = f_I(u^2) = 0$. This is equivalent to the statement $\fix{X}{u} = \emptyset$ by equation \ref{eq:endpointsToStruts}.
\end{proof}

\subsection{Conjugacy Class} \label{sec:conjgacy}
In Section \ref{sec:fixed struts}, we used Theorem \ref{thm:important}, which requires us to determine $\cC{W}{w}$ and $\cC{W_I}{w}$. This means we need to know the sizes of the centralisers of conjugacy classes in $W$ and $W_I$. By Lemma \ref{lem:centraliser conjugacy}, this is equivalent to knowing the size of the conjugacy classes, so we give both here. 

A parabolic subgroup of $W$ will always be a direct product of Coxeter groups \cite{humphrey}, and in particular a parabolic subgroup of any of the $\alpha$, $\beta$, or $H$ series will be a direct product of groups in these series. Therefore, we need only determine the conjugacy classes for the $\alpha$, $\beta$, or $H$ series. 

\subsubsection{Conjugacy Classes of $\alpha_n$}\label{sec:alpha conj}
It has been shown that $\alpha_n \cong S_{n+1}$, the symmetric group on $n+1$ objects \cite[p.~33]{wilson}. It is also well known that the conjugacy classes of $S_{n+1}$ correspond to \emph{cycle-types}, where the cycle-type of an element of $S_n$ is the number of cycles of each length in that permutation. Letting $a_j$ be the number of cycles of length $j$ in a permutation $u$, we have the following equations \cite{wilson}:
\begin{align} 
    \cconj{S_{n}}{u} = \frac{n!}{\prod_{j=1}^n j^{a_j}a_j!},  \\
    \cC{S_{n}}{u} = \prod_{j=1}^n j^{a_j}a_j!  \label{eq:alpha}
\end{align}

\subsubsection{Conjugacy Classes of $\beta_n$}
It has been shown that $\beta_n \cong S_2 \wr S_n$, where $\wr$ is a wreath product \cite[p.~15]{wilson}. A wreath product is special kind of semi-direct product, in this case the product $S_2^n \ltimes S^n$, with multiplication defined as follows:
\begin{align*}
    & \beta_n \cong (S_{2}^n,S_n) &\\
    & ((b_i)_{i=1}^{n},v)  ((a_i)_{i=1}^{n},u) = ((a_i  b_{u(i)})_{i=1}^{n}, v  u) &
\end{align*} 
We define the \emph{flip parity}\footnote{The word flip is used because we want to distinguish this from reflections in general.}, $f: \beta_n \rightarrow S_2$. If $x=((a_i, a_n),u)$ where $u$ is a cycle on the set $K$, let $f(x)=\prod_{k \in K} a_k$.

\begin{theorem}[Conjugacy of Equal Flip Parity Cycles]\label{thm:cefpc}
Let $x=((a_i),u)$ and $z=((b_i),v)$ where $u$ and $v$ are cycles. $x \sim z$ if and only if $v \sim u$ and $f(z)=f(x)$. 

\end{theorem}
\begin{proof}
    To show that $v \sim u$ and $f(z)=f(x)$ implies $x \sim z$, we construct an element which maps $x$ to $z$ under conjugation. Let $y = ((s_i)_{i=1}^{n}, \id)$ where
    \[
        s_i = \prod^{i-1}_{j=1} a_j .
    \]
    Note that $y^{-1} = y$. We now conjugate $x$ by $y$:
    \[
        y x y^{-1} = y x y = y \big((s_i a_i)_{i=1}^{n}, u\big) = \big((s_i a_i s_{u(i)})_{i=1}^{n}, u\big).
    \]
    We must consider the following cases.
    
    \noindent \emph{Case 1:} $i \notin K$. Here,
    \[
        s_i a_i s_{u(i)} = s_i a_i s_i = a_i = 1.
    \]
    
    \noindent \emph{Case 2:} $i \in K$ and $i < N$. In this case,
    \[
        s_i a_i s_{u(i)} = s_i a_i s_{i + 1} = \left(\prod^{i-1}_{j=1} a_j\right) a_i \left(\prod^{i}_{j=1} a_j\right) = a_i a_i = 1.
    \] 

    \noindent \emph{Case 3:} $i \in K$ and $i = N$. Finally, in this case
    \[
        s_i a_i  s_{u(i)} = s_N a_i s_1 = \left(\prod^{N-1}_{j=1} a_j\right) a_N 1 = \prod^{N}_{j=1} a_j = f(x).
    \]

    Therefore, if $f(x)=1$ then $x \sim ((1)_{i=1}^{n}, u)$, and if $f(x)=-1$ then $x \sim ((-1 \text{ if } i=N, 1 \text{ if } i \neq N)_{i=1}^{n}, u)$. Therefore if $v \sim u$ and $f(x)=f(y)$ then $x \sim z$.\\

    To the converse, it follows immediately from $x \sim z$ that that $v \sim u$. Moreover, if $y$ is the element such that $y x y^{-1}=z$, then since $f$ is a homomorphism, we have that 
    \[
        f(z) = f(y^{-1}) \cdot f(x) \cdot f(y) = f(y^{-1} \cdot y) \cdot f(x) = f(x) \implies f(z) = f(x).
    \]
\end{proof}
    
We can decompose the $S_n$ component of each element into a product of disjoint cycles $u \in S_n$ \cite{wilson}. Therefore, elements of $\beta_n$ can be decomposed into disjoint cycles (including 1-cycles) of $S_n$, each with an associated flip parity. By Theorem \ref{thm:cefpc}, two elements of $\beta_n$ are conjugate if and only if their decomposition contains the same number of each cycle-type and flip parity combination.

We now calculate the size of the conjugacy class and centraliser of an arbitrary element $w = ((a_i)_{i=1}^{n},u)$. We already have a formula for $\cconj{S_n}{u}$, which we multiply by the possible combinations of flip parities. 

Let the total number of cycles of length $j$ in $w$ be $A_j$ and the total number of positive flip parity cycles of length $j$ be $B_j$. Hence there are $\binom{A_j}{B_j}$ choices of flip parities.

Each element $i$ in a cycle corresponds to an $a_i \in S_2$. Therefore there are $2^j$ choices for a cycle on any choice of $j$ elements. Half of the choices are positive flip parity and half are negative, so there are $2^{j-1}$ choices given a flip parity.

Therefore, the following expression is the number of distinct elements with a given cycle type:
\[
    \prod_{j=1}^n \binom{A_j}{B_j} (2^{j-1})^{A_j}. \\
\]

Hence we have the following derivation for $\cconj{\beta_n}{w}$.

\begin{align*}
    \cconj{S_n}{u} \prod_{j=1}^n \binom{A_j}{B_j} (2^{j-1})^{A_j} &= n!\left(\prod_{j=1}^n \frac{1}{j^{A_j}A_j!} \right) \left( \prod_{j=1}^n \frac{A_j!}{B_j! (A_j - B_j)!} (2^{j-1})^{A_j} \right) &\\
    & = n! \prod_{j=1}^n \frac{1}{B_j! (A_j - B_j)!} \left(\frac{2^{j-1}}{j} \right)^{A_j} .
\end{align*}

By Lemma \ref{lem:centraliser conjugacy}, we have the following expression for $\cC{\beta_n}{w}$:
\begin{equation}\label{eq:beta conj}
    \cC{\beta_n}{w} = 2^n\prod_{j=1}^n B_j! (A_j - B_j)! \left(\frac{j}{2^{j-1}} \right)^{A_j}.
\end{equation}

\subsubsection{Conjugacy Classes of $H_3$} \label{subsec:hconj}

According to Wilson \cite[p~.33]{wilson} $H_3 \cong A_5 \times S_2$, where $A_5$ is the alternating group on five letters. The conjugacy classes of $A_5$ are the subset of $S_5$ consisting of the even cycle types \cite{gallian}. We can find the size of the conjugacy classes of $A_5$ simply using the formula for those of $S_5$. To obtain the conjugacy classes of $H_3$, we just create two copies for each class of $A_5$, corresponding to the two choices of the $S_2$ factor.

\subsubsection{Conjugacy Classes of a Direct Product of Groups}
A direct product can contain any combination of conjugacy classes of the constituent groups. We simply add the frequency of each conjugacy class together, and count any new conjugacy classes formed by combining conjugacy classes of each group together.

Now that we know the size of the conjugacy classes, we can calculate the number of points fixed by each group element.

\subsection{Orbit Counting Theorem}
We now use Burnside's lemma calculate the number of orbits of the struts from the number of struts fixed by each $u \in W$. 

\begin{lemma}[Burnside's Lemma]\label{lem:burnside}
    The following equation can be used to compute the number of orbits of a group action on a set \cite{gallian}:
    \[
        \card{X/G} = \frac{1}{\card{G}}\sum_{g \in G} \card{\fix{X}{g}}.
    \]
\end{lemma}
Applying Lemma \ref{conjugacy}, we can split the sum up by conjugacy class. Let $u_i$ be a representative of the $i$th conjugacy class of $W$.

\begin{lemma}[Orbit Counting by Conjugacy] \label{lemma:modburnside}
    Since struts are considered distinct precisely when they are not equivalent under $W$, we can calculate the number of distinct seed struts as
    \[
        \card{X/W} = \frac{1}{\card{W}}\sum_i \card{\fix{X}{u_i}} \cC{W}{u_i} = \sum_i \frac{\card{\fix{X}{u_i}}}{\cC{W}{u_i}}.
    \]
\end{lemma}

We can substitute in our equations for $\card{\fix{X}{u_i}}$ from Theorem \ref{theorem:fixedstruts}.
\begin{theorem}
    Now let $u_i$ be a representative of the $i$th conjugacy class of $W_I \cap W_J$. In the heteroorbital case, we obtain the following sum by substituting for $\fix{X}{u_i}$ using Theorem \ref{theorem:fixedstruts}:
\[
    \card{X/W}= \sum_i \frac{\cC{W}{u_i}}{\cC{W_I}{u_i} \cC{W_J}{u_i}}.
\]

For homoorbital types we have two sums: one for conjugacy classes which intersect $W_I$ and the other for all those which intersect $\sqrt{W_I}$. Let $u_i$ be a representative of the $i$th conjugacy class of $W_I$ and let $w_j$ be a representative of the $j$th conjugacy class of $\sqrt{W_I}$. We have
\[
    \card{X/W}= \sum_i \left[\frac{\cC{W}{u_i}}{2\cC{W_I}{u_i} \cC{W_I}{u_i}} - \frac{1}{\cC{W_I}{u_i}} \right] 
    + \sum_j \frac{\cC{W}{w_j^2}}{2\cC{W_I}{w_j^2}\cC{W}{w_j}}.
\]
We can simplify the second term because \[
\card{X/W}= \sum_i \frac{\card{W_I}}{\cC{W_I}{u_i}} = \sum_i \cconj{W_I}{u_i} = \card{W_I} \implies \sum_i \frac{1}{\cC{W_I}{u_j}} = 1.
\] 
Hence, we can rewrite the expression for the number of distinct struts as
\[
\card{X/W}=\frac{1}{2} \left[ \sum_i \frac{\cC{W}{u_i}}{\cC{W_I}{u_i}^2} 
+ \sum_j \frac{\cC{W}{w_j^2}}{\cC{W_I}{w_j^2}\cC{W}{w_i}} \right] - 1.
\]
\end{theorem}

\subsection{Univalent polypolyhedra}

In this section we will eliminate the number of polypolyhedra which do not satisfy property 2 of the definition of polypolyhedra in Section \ref{key_concepts} by checking the valency, $d$, of the two endpoints, $C_I$ and $C_J$. If either endpoint has valency less than 2, then it is not an intersection point, and so we eliminate that seed strut. We shall refer to the seed strut as $x$ and the set of all struts as $X$. We will use the following two consequences of having a univalent endpoint, which without loss of generality we let be $C_I$, to deduce when this can occur.

\begin{lemma} \label{lem:handshake}
If $d(C_I) = 1$ and $I=J$, then $2\card{X} = [W:W_I]$. But if $d(C_I) = 1$ and $I \neq J$, then $\card{X}=[W:W_I]$.
\end{lemma}
\begin{proof}
    The second condition follows from the handshaking lemma, $2\card{E} = \sum_{v \in V} d(v)$. Letting the struts, $X$, play the part of edges and endpoints that of vertices, we have that if $I=J$ then $2\card{X} = [W:W_I]d(C_I)$. But if $I \neq J$ then $\card{X}=[W:W_I]d(C_I)$. The lemma follows by substituting $1$ for $d(C_I)$.
\end{proof}
    
We will now apply these lemmas to determine how many distinct polypolyhedra have univalent seed struts there are of each orbit-type . We shall assume that $I \neq \emptyset$ and $J \neq \emptyset$. The case where $I = \emptyset$ is dealt with in Lemma \ref{lem:univalent C type}.

\begin{lemma}\label{lem:univalent}
    A seed strut has univalent endpoints precisely when there exists some  $w \in W \backslash W_I$ such that $wIw^{-1} \subset S$.
\end{lemma}
\begin{proof}
    Suppose $C_I$ is a univalent endpoint. We have
    \[[W:\stab(x)] = [W^* : \stab(x)^*] = \card{X}.\]
    Then, by Lemma \ref{lem:handshake}, in the homoorbital case we have,
    \[2[W:\stab(x)] = [W:W_I] \implies [W_I:\stab(x)]=2\]
    And in the heteroorbital case we have,
    \[[W:\stab(x)] = [W:W_I] \implies W_I = \stab(x)\]
    Because $W_I \leq \stab(x)$ we have that $x$, along with both of its endpoints, is contained in the axis $\bigcap_{i \in I} H_i$. Therefore $wIw^{-1} = J \subset S$ for some $w \in W$. Moreover, if $I = J$, since $2[W:\stab(x)] = [W:W_I]$ and $W_I < \stab(x)$, we have $\stab(x) \cong S_2 \times W_I$. The $S_2$ component corresponds to a flip of the axis, which is a non-trivial element of $W \backslash W_I$. 
\end{proof}

\begin{definition}[Dual Coxeter Group]
Define the \emph{dual} of a Coxeter group $W$ to be the the image of the map $s_i \mapsto s_{n-i+1}$. We denote the dual of a subgroup $H$ by $\overline{H}$. 

The dual map is an isomorphism of graphs from $G(W)$ to its mirror image. Clearly, the dual map is a bijection, but it is only a homomorphism---hence automorphism---if $m(i,j)=m(n-i+1,n-j+1)$. Equivalently, the dual map is an autormphism of $G(W)$. Precisely when the dual map is an automorphism we shall say that $G$ is self-dual. For example, \dynkin[Coxeter] A3 is self-dual while \dynkin[Coxeter] B3 is not.
    
\end{definition}

\begin{theorem}\label{thm:univalent}
    There is exactly one seed strut with a univalent endpoint precisely when $\Z(W)$ is non-trivial and $I=J$, or $\Z(W)$ is trivial and $I = \overline{J}$.
\end{theorem}
\begin{proof}
Suppose a $(W,I,J)$ seed strut has univalent endpoints. By Lemma \ref{lem:univalent}, there is some $w \in W \backslash W_I$ such that $wIw^{-1} \subset S$. Let us denote the automorphism $u \mapsto wuw^{-1}$ by $\phi$.

We first prove that $\phi(S)=S$. Since $\card{I}=n-1$, let $s_k$ denote the single element of $S \backslash I$ and $s_l$ denote the single element of $S \backslash \phi(I)$. Since 
\[(\phi(s_k)\phi(s_i))^{m(k,i)}=(\phi(s_l)\phi(s_i))^{m(k,i)}=1 \text{ for each } 1 \leq i \leq n,\] 
we have $\phi(s_k)s_l = \phi(s_k)\phi(s_k) = 1$, and hence $\phi(s_k) = s_l \in S$ and $\phi(S) = S$. Hence $\phi$ is a graph automorphism of $G(W)$, meaning $\phi$ either fixes $S$ pointwise or maps $S$ to $\overline{S}$. 

    \noindent \emph{Case 1:} $\Z(W)$ is non-trivial. Therefore $\phi$ fixes $S$ pointwise, so $I=J$ (homoorbital). We have $wsw^{-1}=s$ for all $s \in S$, so $w \in \Z(W)$. Since $\id \in W_I$, $\Z(W)$ must be non-trivial. 
    
    Conversely, if $w \in \Z(W) \backslash \{\id\}$ then $w \in \stab(x)$ and if $I = J$ then $W_I < \stab(x)$, so $w \in \stab(x) \backslash W_I$. Hence in this case a seed strut has univalent endpoints precisely when $I=J$.

    \noindent \emph{Case 2:} $\Z(W)$ is trivial. Of the groups in Table $\ref{tab:irr class}$ with trivial centre, all are self-dual, so we shall assume that $W$ is self-dual. Therefore $\phi$ maps $S$ to its dual, and by extension $I = \phi(J) = \overline{J}$. 
    
    Conversely, if $I = \overline{J}$ and $G(W)$ is self-dual, then we can construct $w$ in the following way. The self-dual groups fall into two classes, $\dih_d$ and $\alpha_n$; we shall construct $w$ explicitly in both cases. If $W=\dih_d$, $d$ odd, then let $w=(s_1 s_2)^{\frac{d-1}{2}}$. We can then confirm that $w$ is dual map,
    \[ws_1w^{-1} = (s_1 s_2)^{\frac{d-1}{2}}s_1(s_2 s_1)^{\frac{d-1}{2}}=(s_1 s_2)^{d-1}s_1 = (s_1 s_2)^{d}s_2 = s_2\].
    
    Since $W = \alpha_n \cong S_{n+1}$, we can represent $s_i \in S$ as the transposition $(i,i+1)$. Then we let $w = \prod_{1 \leq i \leq \frac{n}{2}}(i,n-i+1)$. Since $w$ is a product of disjoint transpositions, $w = w^{-1}$. We confirm that conjugation by $w$ is equivalent to the dual map by taking an arbitrary element $s_k \in S$.  
    \begin{align*}
        & ws_k w^{-1} = ws_k w = \big(\prod_{1 \leq i \leq \frac{n}{2}}(i,n-i+1) \big) (k, k+1) \big( \prod_{1 \leq i \leq \frac{n}{2}}(i,n-i+1) \big) \\
        &= (k, n-k+1)(k+1, n-k)(k, k+1)(k, n-k+1)(k+1, n-k)  \\
        &= (n-k+1, n-k) = \overline{s_k}.
    \end{align*}

    Hence in this case a seed strut has univalent endpoints precisely when $I = \overline{J}$. 
\end{proof}
The groups with from Table \ref{tab:irr class} with non-trivial centre are $\beta_n$, $H_3$, and $\dih_d$ where $d$ is even; the only groups with trivial centre are $\alpha_n$ and $\dih_d$ where $d$ is odd \cite{coxeter}. We can use theorem \ref{thm:univalent} to determine for which orbit-types there is precisely one seed strut with univalent endpoints and for which there are none. 

\subsubsection{$\emptyset$ Orbit-types}
There can be more univalent orbit-types when $I = \emptyset$. Most notably, any heteroorbital type.

\begin{lemma}\label{lem:univalent C type}
    If $I= \emptyset$ then all seed struts for which $I \neq J$ have a univalent endpoint and $U(W^*,\emptyset,\emptyset)$ is equal to the number of order 2 elements of $W^*$.
    
\end{lemma}
\begin{proof}
    \emph{Case 1:} $I \neq J$. Then if $w \in \stab(x)$, $w$ cannot transpose the endpoints of $x$, so $w$ must fix both of endpoints. One such endpoint is $C_{\emptyset}$, so $w = \id$. Therefore $\stab(x) = \{\id\}$. Hence $\card{X} = [W^*: \{\id\}] = \card{W^*}$. By Lemma \ref{lem:handshake} we have, 
    \[
    \card{X} = [W^*:W_{\emptyset}]d(C_{\emptyset}) \implies d(C_{\emptyset})= \frac{\card{X}}{[W^*:W_{\emptyset}]}= \frac{\card{W^*}}{\card{W^*}}=1.
    \] 
    Hence $C_\emptyset$ is univalent. 

    \emph{Case 2:} $I = J$. By Lemma \ref{lem:handshake}, 
    \[
    2 [W^*:\stab(x)^*]=[W^*:W_\emptyset]d(C_\emptyset) \implies d(C_\emptyset) = \frac{2}{\card{\stab(x)^*}}.
    \]
    Therefore $C_\emptyset$ is univalent precisely when $\card{\stab(x)^*} = 2$. This is the same as a strut being flipped, as in Section \ref{sec:fixed struts}. Therefore we know that $U(W^*, \emptyset, \emptyset)$ is equal to the number of order $2$ elements of $W^*$.
\end{proof}

\section{Results}\label{sec:main result}

We now have a pair of formulae for the number of distinct polypolyhedra for a given orbit-type, $(W, I, J)$ in terms of the centralisers of elements in $W$, $W_I$, and $W_J$. Let $U(W,I,J) \in \{0,1\}$ denote the number of univalent struts, which we determine using Theorem \ref{thm:univalent}. We calculate $P(W,I,J)$ as $\card{X/W} - U(W,I,J)$.

\begin{theorem}[Main Result]\label{thm:main result}
Let $u_i$ be a representative of the $i$th conjugacy class of $W_I \cap W_J$. For the heteroorbital case, we have
\begin{equation*}
    P(W,I,J) = \sum_i \frac{\cC{W}{u_i}}{\cC{W_I}{u_i} \cC{W_J}{u_i}} - U(W,I,J).
\end{equation*}
Let $u_i$ be a representative of the $i$th conjugacy class of $W_I$ and let $w_j$ be a representative of the $j$th conjugacy class of $\sqrt{W_I}$. In the homoorbital case, we have
\begin{equation*}
    P(W,I,J) = \frac{1}{2} \left[ \sum_i \frac{\cC{W}{u_i}}{\cC{W_I}{u_i}^2} 
+ \sum_j \frac{\cC{W}{w_j^2}}{\cC{W_I}{w_j^2}\cC{W}{w_j}} \right] - 1 - U(W,I,J).
\end{equation*}
\end{theorem}

\subsection{Examples}

We will now show some worked examples of Theorem \ref{thm:main result} before presenting the number of polypolyhedra for a selection of rank 3 groups on which Lang performed his enumeration \cite{lang_2016}.

\begin{example}[$\alpha_4, \{s_1,s_2,s_3\}, \{s_1,s_2,s_4\}$ ]
This is a heteroorbital type, with $W= \alpha_4$ $W_I = \alpha_3$, and $W_J= \alpha_2 \times S_2$. $G(W)$ is \dynkin[Coxeter] A4, $G(W_I)$ is \dynkin[Coxeter] A3, and $G(W_J)$ is \dynkin[Coxeter] A2 \dynkin[Coxeter] A1.

By proposition \ref{prop:shortcut} we need only consider the conjugacy classes of $\alpha_3 \cap (\alpha_2 \times S_2)$. These are equivalent to $S_4$ and $S_3 \times S_2$, so our conjugacy classes are either in $S_3$ or $S_2 \times S_2$. 
    
In Table \ref{tab:ex alpha}, we choose representatives of each conjugacy class in $\alpha_3 \cap (\alpha_2 \times S_2)$ and calculate the size of the centralisers of these representatives using Equation \ref{eq:alpha}. 

\begin{table}[ht]
\begin{center}
\begin{tabular}{|c|c|c|c|c|}
\hline
Class Rep., $w$ & $\id$ & $(12)$ & $(123)$ & $(12)(34)$ \\
\hline
$\C{W}{w}$                     & 120   & 12     & 6       & 8          \\
\hline
$\C{\alpha_3}{w}$            & 24    & 4      & 3       & 8          \\
\hline
$\C{\alpha_2 \times \alpha_1}{w} $            & 12    & 3      & 6       & 4 \\
\hline
\end{tabular}
\caption{Table of centralisers in $\alpha_4$.}\label{tab:ex alpha}
\end{center}
\end{table}

The equation for the number of polypolyhedra of a heteroorbital type is given by
\[
\sum_i \frac{\cC{W}{u_i}}{\cC{W_I}{u_i} \cC{W_J}{u_i}} - U(W,I,J).
\]

Since $\alpha_4$ has trivial centre and $I \neq \overline{J}$, by Theorem \ref{thm:univalent} there can be no univalent endpoints, so $U(W,I,J)=0$. Evaluating the sum we obtain
\[
    P(\alpha_4, \{s_1, s_2, s_3\}, \{s_1, s_2, s_4\}) = \frac{120}{12 \cdot 24} + \frac{12}{3 \cdot 4} + \frac{6}{6 \cdot 3} + \frac{8}{4 \cdot 8} - 0 = 2.
\]
Thus there are precisely 2 non-univalent polypolyhedra with orbit-type $(\alpha_4, \{s_1,s_2,s_3\}, \{s_1,s_2,s_4\})$.
      
\end{example}

\begin{example}[$H_3, \{s_1, s_2\}, \{s_1, s_2\}$]
    This is a homoorbital type, with $W= H_3$ and $W_I = \dih_5$. $G(W)$ is \dynkin[Coxeter] H3 and $G(W_I)$ is \dynkin[Coxeter,gonality=5] I{}.

    By proposition \ref{prop:shortcut} we need only consider the conjugacy classes of $\sqrt{\dih_5}$. It turns out that $\sqrt{\dih_5} = \dih_5$. In Table \ref{tab:ex H}, we choose representatives of each conjugacy class and calculate the size of the centralisers of these representatives using Theorem \ref{subsec:hconj}.

    \begin{table}[ht]
    \begin{center}
    \begin{tabular}{|c|c|c|c|c|}
    \hline
    Class Rep., $w$ & $\id$ & $(12)(23)$ & $(12345)$ & $(13524)$ \\
    \hline
    $\C{W}{w}$                      & 120   & 8          & 10        & 10        \\
    \hline
    $\C{\dih_5}{w}$             & 10    & 2          & 5         & 5         \\
    \hline
    $w^2$                       & $\id$ & $\id$      & $(13524)$ & $(12345)$ \\
    \hline
    $\C{W}{w^2}$                    & 120   & 120        & 10        & 10        \\
    \hline
    $\C{ \dih_5 }{w^2}$             & 10    & 10         & 5         & 5        \\
    \hline
    \end{tabular}
    \caption{Table of centralisers in $H_3$.}\label{tab:ex H}
    \end{center}
    \end{table}

    The equation for the number of polypolyhedra is
\[
    \frac{1}{2} \left[ \sum_i \frac{\cC{W}{u_i}}{\cC{W_I}{u_i}^2} 
+ \sum_j \frac{\cC{W}{w_j^2}}{\cC{W_I}{w_j^2}\cC{W}{w_j}} \right] - 1 - U(W,I,J).
\]
    Since $H_3$ has non-trivial centre and $I=J$, by Theorem \ref{thm:univalent}, there is precisely one possible polypolyhedron with univalent endpoints, so $U(W,I,J)=1$. Evaluating the sum we obtain,
\begin{align*}
    & P(H_3, \{s_1, s_2\}, \{s_1, s_2\}) & \\
    & = \frac{1}{2} \left[ \frac{120}{10^2} + \frac{8}{2^2} + \frac{10}{5^2} + \frac{10}{5^2} + \frac{120}{10 \cdot 120} + \frac{120}{10 \cdot 8} + \frac{10}{5 \cdot 10} + \frac{10}{5 \cdot 10}\right] - 1 - 1 = 1.
\end{align*}

Thus there is precisely 1 non-univalent polypolyhedron with orbit-type $(H_3, \{s_1, s_2\}, \{s_1, s_2\})$.
\end{example}

\begin{example}[$\beta_3, \emptyset, \emptyset$]
This is a homoorbital type, with $W= \beta_3$ and $W_I = \{ \id \}$. 

By proposition \ref{prop:shortcut} we need only consider the conjugacy classes of $\sqrt{\{\id\}}$. In Table \ref{tab:ex beta}, we choose representatives of each conjugacy class and calculate the size of the centralisers of these representatives using Equation \ref{eq:beta conj}. 

\begin{table}[ht]
\begin{center}
\begin{tabular}{|c|c|c|c|c|c|c|}
\hline
Class Rep., $w$ & $\id$ & $(000,(12))$ & $(001,(12))$ & $(001, \id)$ & $(011, \id)$ & $(111, \id)$ \\
\hline
$\C{W}{w}$                & 48    & 8            & 8            & 16           & 16           & 48           \\
\hline
$\C{\{\id\}}{w}$          & 1     & 0            & 0            & 0            & 0            & 0            \\
\hline
$w^2$                     & $\id$ & $\id$        & $\id$        & $\id$        & $\id$        & $\id$        \\
\hline
$\C{W}{w^2}$              & 48    & 48           & 48           & 8            & 8            & 16           \\
\hline
$\C{\{\id\}}{w}$          & 1     & 1            & 1            & 1            & 1            & 1            \\
\hline
\end{tabular}
\caption{Table of centralisers in $\beta_3$.}\label{tab:ex beta}
\end{center}
\end{table}

By Lemma \ref{lem:univalent C type}, $U(\beta_3,\emptyset, \emptyset)$ is equal to sum of the number of elements in each conjugacy class with order 2. Therefore we have
\begin{align*}
    & U(\beta_3,\emptyset, \emptyset) = \cconj{\beta_3}{(000,(12))} + \cconj{\beta_3}{(001,(12))} \\
    & + \cconj{\beta_3}{(001,\id} + \cconj{\beta_3}{(011,\id)} + \cconj{\beta_3}{(111,\id)} \\
    &= 8 + 8 + 3 + 3 + 1 = 23
\end{align*}
\[
    \frac{1}{2} \left[ \sum_i \frac{\cC{W}{u_i}}{\cC{W_I}{u_i}^2} + \sum_j \frac{\cC{W}{w_j^2}}{\cC{W_I}{w_j^2}\cC{W}{w_j}} \right] - 1 - U(W,I,J).
\]
Evaluating this sum we obtain,
\[
    P(\beta_3, \emptyset, \emptyset) = 
    \frac{1}{2} \left[ \frac{48}{1^2} + \frac{48}{48} + \frac{48}{8} + \frac{48}{8} + \frac{48}{16} + \frac{48}{16} + \frac{48}{48} \right] - 1 - 23 = 13.
\]

Thus there are precisely 23 non-univalent polypolyhedra with orbit-type $(\beta_3,\emptyset, \emptyset)$.

\subsection{Number of polypolyhedra with groups of rank 3}

Table \ref{tab:rank 3 non-uni} lists the number of non-univalent polypolyhedra of each orbit-type where $W$ is one of $\dih_6$, $\alpha_3$, $\beta_3$, or $H_3$. Tables \ref{tab:rank 3 raw} and \ref{tab:rank 3 uni} list the total number of polypolyhedra and number of univalent polypolyhedra respectively. 

The data in Table \ref{tab:rank 3 raw} corresponds to the values which Lang's programme computed to obtain the total number of $188$ distinct possible polypolyhedra across these groups \cite[p.~8]{lang_2016}. These are calculated using Theorem \ref{theorem:fixedstruts}. In order for our table to comparable with Lang's work, we had to account for the fact that his $\emptyset$ orbits are computed using only $W^*$, which means that they have half as many elements \footnote{If $I \neq \emptyset$, $[W_I:W_I^*] = [W:W^*]=2$, so $W^*/W_I^* = W/W_I$. In contrast, $[\emptyset: \emptyset^*]=1$, so $2\card{W^*/W_\emptyset} = \card{W/W_\emptyset}$.}. 

We also had to add rows for what Lang refers to as \emph{quasi-homoorbital} types, which describe struts whose endpoints have the same stabiliser, but due to being at different radial distances from the origin, are not in the same orbit. Quasi-homoorbital types have been omitted from the present paper because they have the same algebraic properties as homorbital types. Consider the map which scales each endpoint to have radial distance $1$. The result is a either a homoorbital polypolyhedron or, when the struts were originally between paired endpoints, no polypolyhedron at all. This map is clearly a bijection to the set of homoorbital polypolyhedra together with the degenerate ``non-polypolyhedron''. Hence the number of quasihomoorbital $(W,I,I)$ polypolyhedra is always precisely one more than the number of homoorbital $(W,I,I)$ polypolyhedra. Thus it is unneccessary to calculate the number of quasihomoorbital polypolyhedra seperately. 

\begin{table}[ht]
\begin{center}
\begin{tabular}{|c|c|c|c|c|}
\hline
                            & $\text{Dih}_6$ & $\alpha_3$ & $\beta_3$ & $H_3$ \\
\hline
$\{s_1,s_2\},\{s_1,s_2\}$ & 3              & 1          & 2         & 3     \\
\hline
(Quasi-homoorbital)         & 4              & 2          & 3         & 4     \\
\hline
$\{s_1,s_2\},\{s_1,s_3\}$ & 3              & 2          & 3         & 5     \\
\hline
$\{s_1,s_2\},\{s_2,s_3\}$   & 1              & 2          & 2         & 4     \\
\hline
$\{s_1,s_3\},\{s_1,s_3\}$   & 3              & 2          & 4         & 8     \\
\hline
(Quasi-homoorbital)         & 4              & 3          & 5         & 9     \\
\hline
$\{s_1,s_3\},\{s_2,s_3\}$   & 1              & 2          & 3         & 7     \\
\hline
$\{s_2,s_3\},\{s_2,s_3\}$   & 1              & 1          & 3         & 5     \\
\hline
(Quasi-homoorbital)         & 2              & 2          & 4         & 6     \\
\hline
$\emptyset, \emptyset$      & 9              & 7          & 16        & 37    \\
\hline
Total                       & 31             & 24         & 45        & 88    \\
\hline
\end{tabular}
\caption{Number of univalent and non-univalent polypolyhedra for Coxeter groups of rank 3.}\label{tab:rank 3 raw}
\end{center}
\end{table}

Table \ref{tab:rank 3 uni} is calculated using Theorem \ref{thm:univalent} for types without $C_\emptyset$ endpoints and Lemma \ref{lem:univalent C type} for types with $C_\emptyset$ endpoints. Table \ref{tab:rank 3 non-uni} is calculated using Theorem \ref{thm:main result}, and this result gives us a final total of $128$ distinct non-univalent polypolyhedra from the four groups we considered.

\begin{table}[ht]
    \centering
    \begin{tabular}{|c|c|c|c|c|}
    \hline
         & $\dih_6$ & $\alpha_3$ & $\beta_3$ & $H_3$ \\
    \hline
    $\{s_1,s_2\},\{s_1,s_2\}$ & 1        & 0          & 1         & 1     \\
    \hline
    $\{s_1,s_2\},\{s_1,s_3\}$ & 0        & 0          & 0         & 0     \\
    \hline
    $\{s_1,s_2\},\{s_2,s_3\}$ & 0        & 1          & 0         & 0     \\
    \hline
    $\{s_1,s_3\},\{s_1,s_3\}$ & 1        & 1          & 1         & 1     \\
    \hline
    $\{s_1,s_3\},\{s_2,s_3\}$ & 0        & 0          & 0         & 0     \\
    \hline
    $\{s_2,s_3\},\{s_2,s_3\}$ & 1        & 0          & 1         & 1     \\
    \hline
    $\emptyset, \emptyset$    & 1        & 9          & 23        & 31    \\
    \hline
    Total                     & 4        & 11         & 26        & 34    \\
    \hline
    \end{tabular}
    \caption{Number of univalent polypolyhedra for Coxter groups of rank 3.}
    \label{tab:rank 3 non-uni}
\end{table}

\begin{table}[ht]
    \centering
    \begin{tabular}{|c|c|c|c|c|}
    \hline
       & $\dih_6$ & $\alpha_3$ & $\beta_3$ & $H_3$ \\
    \hline
    $\{s_1,s_2\},\{s_1,s_2\}$ & 2        & 1          & 1         & 2     \\
    \hline
    $\{s_1,s_2\},\{s_1,s_3\}$ & 3        & 2          & 3         & 5     \\
    \hline
    $\{s_1,s_2\},\{s_2,s_3\}$ & 1        & 1          & 2         & 4     \\
    \hline
    $\{s_1,s_3\},\{s_1,s_3\}$ & 2        & 1          & 3         & 7     \\
    \hline
    $\{s_1,s_3\},\{s_2,s_3\}$ & 1        & 2          & 3         & 7     \\
    \hline
    $\{s_2,s_3\},\{s_2,s_3\}$ & 0        & 1          & 2         & 4     \\
    \hline
    $\emptyset, \emptyset$    & 5        & 7          & 12        & 44    \\
    \hline
    Total                     & 14       & 15         & 26        & 73    \\
    \hline
    \end{tabular}
    \caption{Number of non-univalent polypolyhedra for Coxter groups of rank 3.}
    \label{tab:rank 3 uni}
\end{table}

\end{example}

\section{Lang's Algorithm}\label{sec:langsmain result}
To verify our main result we compare it with Lang's main result, which is based on the implicit assumption that two struts of the same orbit-type are equivalent if they have the same length. Our analysis in this section shows that this assumption is not \emph{in general} true, but as we show, it is by serendipity true for all groups of rank 3 (although this would likely not have been known to Lang).

Here we explain the main step of Lang's algorithm in order to point out where the implicit assumption is. If our orbit type is $(W,I,J)$, we first choose any point $p \in \mathbb{R}^3$ of type $I$, which without loss of generality we may take to be $C_I$. Let $Q$ denote the set of points of type $J$, $\{wC_J \mid w \in W\}$. Let $(p,Q)=\{(p,q) \mid q \in Q\}$, which represents the set of all possible struts incident $p$. Given any strut $(p',q)$ where $p'$ is a $C_I$ point and $q$ is a $C_J$ point, there must be a $w \in W$ such that $w(p')=p$. Since $w(q) \in Q$, it must be the case that $(w(p'),w(q)) \in (p,Q)$. Hence, $(p,Q)$ contains a representative of every orbit of the struts. 

Recall from Section \ref{sec:cox} that Coxeter groups of rank $n$ have a representation as subgroups of $O(n)$, which preserve Euclidean distance. Therefore, if two struts are in the same orbit then their lengths are equal. Lang's algorithm computes the list of the lengths of each potential seed strut in $(p,Q)$ and considers two struts to be equivalent precisely when they have the same length. Since each orbit of struts has a representative in $(p,Q)$, each possible length of strut is represented.

However, in order for distinct struts not to be considered equivalent, this approach relies on the implicit assumption that struts of equal length are always in the same orbit. The following counter-example demonstrates that this is not necessarily so.

Let $W = \beta_4$ (\dynkin[Coxeter] B4) and $I=J=\{s_1, s_3, s_4\}$ (\dynkin[Coxeter] A1 \dynkin[Coxeter] B2). Since $\beta_4 \cong S_2 \wr S_4$, we can represent $\beta_4$ as a matrix group in two steps. First, we take a permutation of the four standard basis vectors $\{e_i\}_{i=1}^4$, namely $(1,0,0,0)$, $(0,1,0,0)$, $(0,0,1,0)$, and $(0,0,0,1)$, encoding the information from $S_4$. Then we multiply this by a diagonal matrix with entries $-1$ or $1$, which gives us a wreath product with $S_2$. One choice of generating reflections is
\begin{align*}
    s_1 = \begin{bmatrix}
        1 & 0 & 0 & 0 \\
        0 & 1 & 0 & 0 \\
        0 & 0 & 0 & 1 \\
        0 & 0 & 1 & 0
    \end{bmatrix}, \quad
    s_2 = \begin{bmatrix}
        1 & 0 & 0 & 0 \\
        0 & 0 & 1 & 0 \\
        0 & 1 & 0 & 0 \\
        0 & 0 & 0 & 1
    \end{bmatrix}, \\
    s_3 = \begin{bmatrix}
        0 & 1 & 0 & 0 \\
        1 & 0 & 0 & 0 \\
        0 & 0 & 1 & 0 \\
        0 & 0 & 0 & 1
    \end{bmatrix}, \quad
    s_4 = \begin{bmatrix}
        -1 & 0 & 0 & 0 \\
        0 & 1 & 0 & 0 \\
        0 & 0 & 1 & 0 \\
        0 & 0 & 0 & 1
    \end{bmatrix}
\end{align*}

The point $C_I$ should be fixed by $s_1$, $s_3$, and $s_4$ and not by $s_2$. We shall represent $C_I$ by $a=(0,0,1,1)$. It is easy to see that the orbit of $C_I$ is precisely the set of vectors $\{ \pm e_i \pm e_j \mid i \neq j\}$. Moreover, a straight-forward check shows that the vectors $b=(1,1,0,0)$ and $c=(0,0,1,-1)$ are equal distance from $a$, and yet the pair $\{a,b\}$ is not in the same orbit $\{a,c\}$ under $W$. Thus we have an example of struts with equal length which are nonetheless not equivalent. This shows that Lang's algorithm is not in general valid.

\section{Conclusion}\label{sec:conc}

We gave a direct enumeration of polypolyhedra, giving a count of $128$. In contrast to the earlier, partially-computational work of Lang, we obtained the number of distinct polypolyhedra by calculating the number of orbits of possible seed struts under a given Coxeter group, $W$, and subtracting the number of univalent polypolyhedra. 

The number of orbits was calculated by using a modified version of Burnside's lemma. This required us to calculate the number of struts fixed by an element of each conjugacy class of $W$, which in turn required calculating the number of each type of endpoint fixed by an element of each conjugacy class. The latter calculation was the subject of the main theorem of our paper, Theorem \ref{thm:important}. This relied on a calculation of the conjugacy classes of $W$ and its parabolic subgroups.

Polypolyhedra which do not meet property 3-5 of Section \ref{key_concepts} are often subjectively as interesting as those which do. However, these classes of polypolyhedra do not meet Lang's original definition of polypolyhedra, so we would at least like to know how many there are. That we were not able to determine this is probably because the precise number of intersections is essentially a topological rather than algebraic property, so there might not be a better approach than Lang's brute force algorithm to enumerate this class of polypolyhedra.

For future consideration, we speculate here that insights from the theory of root systems, which is very well developed, could be applied to polypolyhedra whose symmetry group is a Weyl group\footnote{Weyl groups are Coxeter groups which are the isometry group of the root system of a Lie algebra.}.

Finally, we note that although polypolyhedra may not themselves be mathematically important, the algebraic approach employed in the present paper---especially our Theorem \ref{thm:important}---can be gainfully applied to other geometrical enumeration problems, such as those arising in applications, for example in chemistry.

\printbibliography

@inbook{coxeter,
  title={Regular polytopes},
  author={Coxeter, Harold Scott Macdonald},
  year={1973},
  chapter={12},
  publisher={Courier Corporation}
}

@book{humphrey,
  title={Reflection groups and Coxeter groups},
  author={Humphreys, James E},
  number={29},
  year={1990},
  publisher={Cambridge university press}
}

@inbook{gallian,
  title={Contemporary abstract algebra},
  author={Gallian, Joseph A},
  year={2021},
  pages={492,159},
  publisher={Chapman and Hall/CRC}
}

@article{coxeter_enum,
  title={The Complete Enumeration of Finite Groups of the Form $R_i^2=(R_i R_j)^{k_{ij}}= 1$},
  author={Coxeter, Harold SM},
  journal={Journal of the London Mathematical Society},
  volume={1},
  number={1},
  pages={21--25},
  year={1935},
  publisher={Oxford University Press}
}

@misc{belcastro,
      title={Symmetric colorings of polypolyhedra}, 
      author={sarah-marie belcastro and Thomas C. Hull},
      year={2015},
      eprint={1512.00271},
      archivePrefix={arXiv},
      primaryClass={math.MG}
}

@misc{lang_2016, 
title={Origami burrs and polypolyhedra, part III}, 
url={https://www.langorigami.com/wp-content/uploads/2016/02/Polypolyhedra\_part\_3.pdf}, 
journal={Lang Origami}, 
author={Lang, Robert J.}, 
year={2016}
}

@inbook{wilson,
  title={The finite simple groups},
  author={Wilson, Robert},
  volume={147},
  chapter={2},
  year={2009},
  publisher={Springer}
}

\end{document}